\documentclass[12pt, reqno]{amsart}
\usepackage{amsmath, amsthm, amscd, amsfonts, amssymb, graphicx, color}
\usepackage[bookmarksnumbered, colorlinks, plainpages]{hyperref}

\input{mathrsfs.sty}

\newtheorem{theorem}{Theorem}[section]
\newtheorem{lemma}[theorem]{Lemma}
\newtheorem{proposition}[theorem]{Proposition}
\newtheorem{corollary}[theorem]{Corollary}
\theoremstyle{definition}

\newtheorem{example}[theorem]{Example}

\theoremstyle{remark}

\numberwithin{equation}{section}
\begin{document}

\title [Some generalized numerical radius inequalities ]{Some generalized  numerical radius inequalities involving  Kwong functions }

\author[M. Bakherad]{ Mojtaba Bakherad}

\address{ Department of Mathematics, Faculty of Mathematics, University of Sistan and Baluchestan, Zahedan, I.R.Iran.}

\email{mojtaba.bakherad@yahoo.com; bakherad@member.ams.org}


\begin{abstract}
We prove several numerical radius inequalities involving positive semidefinite matrices via the Hadamard product and Kwong functions. Among other inequalities, it is shown that if   $X$ is a arbitrary $n\times n$ matrix and $A,B$ are positive semidefinite, then
  \begin{align*}
\omega(H_{f,g}(A))\leq k\, \omega(AX+XA),
 \end{align*}
which is equivalent to
\begin{align*}
\omega\big(H_{f,g}(A,B)&\pm H_{f,g}(B,A)\big)\\&\leq k'\,\left\{\omega((A+B)X+X(A+B))+\omega((A-B)X-X(A-B))\right\},
\end{align*}
 where  $f$ and $g$ are two continuous functions on $(0,\infty)$ such that $h(t)={f(t)\over g(t)}$ is Kwong, $k=\max\left\{{f(\lambda)g(\lambda)\over \lambda}: {\lambda\in\sigma(A)}\right\}$ and $k'=\max\left\{{f(\lambda)g(\lambda)\over \lambda}: {\lambda\in\sigma(A)\cup\sigma(B)}\right\}$.
\end{abstract} \maketitle
\section{Introduction}
\noindent Let $\mathcal{M}_n$ be the $C^*$-algebra of all
$n\times n$ complex matrices and  $\langle\,\cdot\,,\,\cdot\,\rangle$ be the standard scalar
product in $\mathbb{C}^n$. A capital letter means an $n\times n$ matrix in  $\mathcal{M}_n$.  For Hermitian matrices $A$ and $B$, we write  $A\geq 0$ if $A$ is positive semidefinite,
$A>0$ if $A$ is positive definite, and $A\geq B$ if $A-B\geq0$.  The numerical radius of $A\in \mathcal{M}_n$ is defined by
\begin{align*}
\omega(A):=\sup\{\mid \langle Ax, x\rangle\mid : x\in \mathbb{C}^n, \parallel x \parallel=1\}.
\end{align*}
It is well known that $\omega(\,\cdot\,)$ defines a norm on $\mathcal{M}_n$, which is equivalent to the usual operator norm $\| \,\cdot\, \|$. In fact, for any $A\in \mathcal{M}_n$, $\frac{1}{2}\| A \|\leq \omega(A) \leq\| A \|$; see \cite{gof}. For
further information about numerical radius inequalities we refer the reader to \cite{gof, YAM}
and references therein. We use the notation $J$ for the matrix whose entries are equal to one.

 The  Hadamard product (Schur product) of two matrices $A, B\in\mathcal{M}_n$  is the
matrix $A\circ B$ whose $(i, j)$ entry is $a_{ij}b_{ij}\,\,(1\leq i,j \leq n)$.
 The Schur multiplier  operator $S_A$ on  $\mathcal{M}_n$ is defined by $S_A(X)=A\circ X\,\,(X\in\mathcal{M}_n)$.
 The induced norm of $S_A$ with respect to the numerical radius norm will be denoted by
 \begin{align*}
 \|S_A\|_\omega=\sup_{X\not=0}{\omega(S_A(X))\over \omega(X)}=\sup_{X\not=0}{\omega(A\circ X)\over \omega(X)}.
\end{align*}
 A continuous real valued function $f$ on an interval $(a,b)\subseteq \mathbb{R}$ is called operator monotone  if $A\leq B$ implies $f(A)\leq f(B)$  for all Hermitian matrices $A, B\in\mathcal{M}_n$ with spectra in $(a,b)$. Following \cite{aud}, a continuous real-valued function $f$ defined on an interval $(a, b)$ with $a>0$ is called a Kwong function if the matrix $K_f=\left({f(\lambda_i)+f(\lambda_j)\over \lambda_i+\lambda_j}\right)_{i,j=1,2,\cdots,n}$
is positive semidefinite for any (distinct) $\lambda_1, \cdots, \lambda_n$   in $(a, b)$. It is easy to see that if $f$ is a nonzero Kwong function, then $f$ is positive and ${1\over f}$ is Kwong. Kwong \cite{kwong1} showed that the set of all Kwong functions on  $(0,\infty)$ is a closed cone and includes all non-negative operator monotone functions on $(0,\infty)$. Also, Audenaert \cite{aud} gave a characterization of Kwong functions by showing that, for given $0\leq a <b$, a function $f$ on an interval $(a, b)$ is Kwong if and only if the function $g(x) =\sqrt{x}f(\sqrt{x})$ is operator monotone on $(a^2,b^2)$.

The Heinz means are defined as $H_{\nu}(a, b)=\frac {a^{1-\nu}b^{\nu}+a^{\nu}b^{1-\nu}}{2}$
for $a, b>0$ and $0\leq \nu \leq 1$. These interesting means interpolate between the
geometric and arithmetic means. In fact, the Heinz inequalities assert that $\sqrt{ab}\leq H_{\nu}(a, b)\leq\frac{a+b}{2}$, where $a, b>0$ and $0\leq \nu \leq 1$. There have been obtained several Heinz type inequalities for Hilbert space
operators and matrices; see \cite{Kaur} and references therein.\\ For two continuous functions $f$ and $g$  on $(0,\infty)$ we denote
 \begin{align*}
H_{f,g}(A,B)=f(A)Xg(B)+g(A)Xf(B) \quad\textrm{and}\quad H_{f,g}(A)=f(A)Xg(A)+g(A)Xf(A),
 \end{align*}
 where  $A,B, X\in\mathcal{M}_n$ such that $A, B$ are positive semidefinite. In particular,
$f(t)=t^\alpha$ and $g(t)=t^{1-\alpha}\,(\alpha\in[0,1])$
\begin{align*}
H_\alpha(A,B)=A^{\alpha}XB^{1-\alpha}+A^{1-\alpha}XB^{\alpha}\quad\textrm{and}\quad
H_\alpha(A)=A^{\alpha}XA^{1-\alpha}+A^{1-\alpha}XA^{\alpha}.
 \end{align*}
A norm $|||\,\cdot\,|||$ on $\mathcal{M}_n$ is called  unitarily invariant if
$|||UAV|||=|||A|||$ for all $A\in\mathcal{M}_n$ and all unitary matrices $U, V\in\mathcal{M}_n$. Let $A, B, X\in\mathcal{M}_n$ such that $A$ and $B$ are positive semidefinite.
 In \cite{ham} it was conjectured a general norm inequality of the Heinz inequality $|||H_{f,g}(A,B)|||\leq\,|||AX+XB|||$,
  where $f$ and $g$ are two continuous functions on $(0,\infty)$ such that $f(t)g(t)\leq t$ and the function $h(t)={f(t)\over g(t)}$ is Kwong. In particular, if $f(t)=t^\alpha$ and $g(t)=t^{1-\alpha}\,(\alpha\in[0,1])$, then we state a Heinz type inequality $|||H_\alpha(A,B)|||\leq\,|||AX+XB|||,$
  where  $A,B, X\in\mathcal{M}_n$ such that $A, B$ are positive semidefinite. For further information, we refer the reader to \cite{bakh,bat} and references therein.

   The numerical radius $\omega(\,\cdot\,)$ is a weakly unitarily invariant   norm on $\mathcal{M}_n$, that is $\omega(U^*AU)=\omega(A)$ for every  $A\in\mathcal{M}_n$ and every unitary $U\in\mathcal{M}_n$. In \cite{agha}, the authors proved a Heinz type inequality for the numerical radius as follows
  \begin{align}\label{main1}
\omega(H_\alpha(A))\leq \, \omega(AX+XA),
 \end{align}
 in which $A, X\in\mathcal{M}_n$ such that $A$ is positive semidefinite. They also showed that the inequality
 $\omega(H_\alpha(A,B))\leq \, \omega(AX+XB)$
   is not true in general.

Our research aim  is to show some numerical radius inequalities  via the Hadamard product and Kwong functions. By using some ideas of \cite{fuji, fuji2} and \cite{ham}, we obtain some extensions and generalizations of inequality \eqref{main1}, which are generalizations of a Hienz type inequality for the numerical radius. For instance, we prove if $A,  X\in\mathcal{M}_n$ such that $A$ is positive semidefinite, then
 \begin{align*}
\omega(H_{f,g}(A))\leq k\, \omega(AX+XA),
 \end{align*}
where $f$ and $g$ are two continuous functions on $(0,\infty)$ such that ${f(t)\over g(t)}$ is Kwong and $k=\max\left\{{f(\lambda)g(\lambda)\over \lambda}:\lambda\in\sigma(A)\right\}$.

\section{main results}
For our purpose we need the following lemmas.

\begin{lemma}\cite[Theorem 3.4]{Zhang1}\label{shour}
$($Spectral Decomposition$)$ Let $A\in\mathcal{M}_n$ with eigenvalues
$\lambda_1, \cdots, \lambda_n$. Then $A$ is normal if and
only if there exists a unitary matrix $U$ such that
\begin{align*}
U^*AU={\rm diag}(\lambda_1, \cdots, \lambda_n).
 \end{align*}
 In particular, $A$ is positive definite if and only if the $\lambda_j\,\,(1\leq j \leq n)$ are positive.
\end{lemma}
\begin{lemma}\cite[Corollary 4]{Okubo}\label{Okubo1}
Let  $A=[a_{ij}]\in\mathcal{M}_n$ be positive semidefinite. Then
\begin{align*}
\|S_A\|_\omega=\max_i a_{ii}.
 \end{align*}
\end{lemma}
\begin{lemma}\label{kit}\cite{HKSH}
Let $X, Y\in\mathcal{M}_n$. Then
\\
$({\rm i})\,\,\omega\left(\left[\begin{array}{cc}
 X&0\\
 0&Y
 \end{array}\right]\right)$
  $= \max \{{\omega(X), \omega(Y)}\};$
\\
$({\rm ii})\,\,{\max\left(\omega(X+Y),\omega(X-Y)\right)\over2} $ $ \leq \omega\left(\left[\begin{array}{cc}
 0&X\\
 Y&0
 \end{array}\right]\right)\leq {\omega(X+Y)+\omega(X-Y)\over2}.$

\end{lemma}

Now, we are in position to demonstrate the first result of this section  by using some ideas of \cite{fuji, fuji2, ham}.
\begin{theorem}\label{main}
 Let $A,B\in\mathcal{M}_n$ be positive semidefinite, $X\in\mathcal{M}_n$, and let $f$, $g$ be two continuous functions on $(0,\infty)$ such that $h(t)={f(t)\over g(t)}$ is Kwong.  Then
\begin{align}\label{hob1}
\omega(H_{f,g}(A))\leq k\, \omega(AX+XA),
 \end{align}
where  $k=\max_{\lambda\in\sigma(A)}\left\{{f(\lambda)g(\lambda)\over \lambda}\right\}$.\\
Moreover,  inequality \eqref{hob1} is equivalent to the inequality
 \begin{align}\label{hob11}
\omega\big(H_{f,g}(A,B)&\pm H_{f,g}(B,A)\big)\nonumber\\&\leq k'\,\left\{\omega((A+B)X+X(A+B))+\omega((A-B)X-X(A-B))\right\},
 \end{align}
  where $k'=\max_{\lambda\in\sigma(A)\cup\sigma(B)}\left\{{f(\lambda)g(\lambda)\over \lambda}\right\}$.
\end{theorem}
\begin{proof}
Assume that $A$ is positive definite. Since
the numerical radius is weakly unitarily invariant, we may assume that A is diagonal
matrix with positive eigenvalues $\lambda_1, \cdots, \lambda_n$.
It follows from ${f\over g}$ is a Kwong function that
 \begin{align*}
Z=[z_{ij}]=\Lambda\left({f(\lambda_i)g^{-1}(\lambda_j)+f(\lambda_j)g^{-1}(\lambda_i)\over \lambda_i+\lambda_j}\right)_{( i,j=1,\cdots, n)}\Lambda
\end{align*}
is positive semidefinite, where $\Lambda={\rm diag}\left(g(\lambda_1),\cdots,g(\lambda_n)\right)$.
 It  follows from  Lemma \ref{Okubo1}  that
 \begin{align*}
\|S_Z\|_\omega=\max_i z_{ii}=\max_i {f(\lambda_i)g(\lambda_i)\over \lambda_i} \leq k
 \end{align*}
 or equivalently, ${\omega(Z\circ X)\over \omega(X)}\leq k\,\,(0\neq X \in\mathcal{M}_n)$. If we put  $E=[{1\over \lambda_i+\lambda_j}]$ and $ F=[f(\lambda_i)g(\lambda_j)+f(\lambda_j)g(\lambda_i)]\in \mathcal{M}_n$, then
  \begin{align*}
\omega(E\circ F \circ X)=\omega(Z\circ X)\leq k\,\omega(X)\qquad(X \in\mathcal{M}_n).
 \end{align*}
  Let  the matrix $C$ be the entrywise inverse of $E$, i.e., $C\circ E=J$. Thus
\begin{align*}
\omega(F\circ X)\leq k\,\omega(C\circ X)\qquad(X \in\mathcal{M}_n)
\end{align*}
   or equivalently
\begin{align}\label{bogh1}
\omega(H_{f,g}(A))=\omega(f(A)Xg(A)+g(A)Xf(A))\leq k\, \omega(AX+XA).
 \end{align}
 Now, if $A$ is positive semidefinite, we may  assume that $A =\left[\begin{array}{cc} A_1&0\\ 0&0\end{array}\right]$, where $A_1\in \mathcal{M}_k\,(k < n)$ is a positive definite matrix. Let $X =\left[\begin{array}{cc} X_1&X_2\\ X_3&X_4\end{array}\right]$,
  where $X_1 \in \mathcal{M}_k$ and $X_4\in\mathcal{M}_{n-k}$. Then we have
  \begin{align}\label{bogh2}
\omega(H_{f,g}(A))\nonumber&=\omega\left(\left[\begin{array}{cc} f(A_1)X_1g(A_1)+g(A_1)X_1f(A_1)&0\\ 0&0\end{array}\right]\right)\nonumber\\&\qquad\qquad\qquad\qquad\qquad(\textrm{by Lemma}\,\ref{kit} ({\rm i}))\nonumber\\&\leq k\,\omega\left(\left[\begin{array}{cc} A_1X_1+X_1A_1&0\\ 0&0\end{array}\right]\right)\qquad(\textrm{by\,} \eqref{bogh1})\nonumber\\&= k\, \omega(A_1X_1+X_1A_1)\qquad(\textrm{by Lemma}\,\ref{kit} ({\rm i}))\nonumber\\&\leq k\, \omega(AX+XA)\qquad(\textrm{by  \,\cite[Lemma 2.1]{omid}}).
 \end{align}
 Hence, we reach   inequality \eqref{hob1}. Moreover, if we replace $A$ and $X$ by $\left(\begin{array}{cc} A&0\\ 0&B\end{array}\right)$ and $\left(\begin{array}{cc} 0&X\\ X&0\end{array}\right)$ in inequality \eqref{hob1}, respectively, then
 \begin{align*}
 \omega&\left(\left[\begin{array}{cc} 0&H_{f,g}(A,B)\\ H_{f,g}(B,A)&0\end{array}\right]\right)
\\&\leq k'\,\omega\left(\left[\begin{array}{cc} 0&AX+XB\\ XA+BX&0\end{array}\right]\right),
 \end{align*}
whence
 \begin{align*}
 &\max\Big\{\omega\big(H_{f,g}(A,B)\pm H_{f,g}(B, A)\big)\Big\}\\&\leq2\,\omega\left(\left[\begin{array}{cc} 0&f(A)Xg(B)+g(A)Xf(B)\\g(B)Xf(A)+f(B)Xg(A)&0\end{array}\right]\right)\\&\qquad\qquad\qquad\qquad\qquad(\textrm{by Lemma}\,\ref{kit} ({\rm ii}))
\\&\leq 2k'\,\omega\left(\left[\begin{array}{cc} 0&AX+XB\\ XA+BX&0\end{array}\right]\right)\qquad(\textrm{by\,} \eqref{bogh2})
\\&\leq k'\,\left(\omega(AX+XB+XA+BX)+\omega(AX+XB-XA-BX)\right)\\& \qquad\qquad\qquad\qquad\qquad(\textrm{by Lemma}\,\ref{kit} ({\rm ii})).
 \end{align*}
 Thus, we have  inequality \eqref{hob11}. Also, if we put $B=A$ in inequality \eqref{hob11}, then we reach inequality \eqref{hob1}.
 \end{proof}
 If we take $f(t)=t^\alpha$ and $g(t)=t^{1-\alpha}$ in Theorem \ref{main} for each $0\leq \alpha\leq 1$, then we get the next result.
\begin{corollary}\cite[Theorem 2.4]{agha}
Let  $A,B\in\mathcal{M}_n$ be positive semidefinite, $X\in\mathcal{M}_n$, and let $0\leq\alpha \leq 1$. Then
\begin{align}\label{hob2}
\omega(H_{\alpha}(A))\leq \, \omega(AX+XA).
 \end{align}
Moreover,  inequality \eqref{hob2} is equivalent to the inequality
 \begin{align*}
\omega\big(H_{\alpha}(A,B)&\pm H_{\alpha}(B,A)\big)\nonumber\\&\leq \,\omega((A+B)X+X(A+B))+\omega((A-B)X-X(A-B)).
 \end{align*}
\end{corollary}
\begin{corollary}
Let $A,B\in\mathcal{M}_n$  be positive semidefinite, $X\in\mathcal{M}_n$, and let $f$ be a non-negative operator monotone function on $[0,\infty)$ such that $f'(0)=\lim_{ x\rightarrow 0^+}f'(x)<\infty$ and $f(0)=0$. Then
\begin{align}\label{hob3}
\omega(f(A)X+Xf(A))\leq f'(0) \,\omega(AX+XA).
 \end{align}
 Moreover,  inequality \eqref{hob3} is equivalent to the inequality
\begin{align*}
\omega(X(f(A)+f(&B))+(f(A)+ f(B))X)\nonumber\\&\leq f'(0) \,\Big(\omega((A+B)X+X(A+B))+\omega((A-B)X-X(A-B))\Big).
 \end{align*}
 \end{corollary}
\begin{proof}
 A function  $g$ is non-negative operator  increasing  on $[0,\infty)$ if and only if  ${t\over g(t)}$ is non-negative operator increasing on $[0,\infty)$; see \cite{mond}. Hence ${t\over f(t)}$ is operator increasing. Then $f(t)\over t$ is decreasing.  If $0\leq x\leq t$, then ${f(t)\over t}\leq{f(x)\over x}$. Now, by taking $x\rightarrow0^+$ we have  ${f(t)\over t}\leq f'(0)$. If we put $g(t)=1\,(t\in[0,\infty))$ in Theorem \ref{main}, it follows from $k=k'\leq f'(0)$ that we get the required result.
\end{proof}
We first cite the following lemma due to Fujii et al. \cite{fuji}, which will be needed in the next theorem.
\begin{lemma}\cite[Lemma 3.1]{fuji}\label{fuji1}
Let $\lambda_1,\cdots, \lambda_n$ be any positive real numbers and $-2<t\leq2$. If $f$ and $g$ are two continuous functions on $(0,\infty)$ such that ${f(t)\over g(t)}$ is Kwong, then the $n\times n$ matrix
\begin{align*}
Y=\left({f(\lambda_i)g^{-1}(\lambda_j)+f(\lambda_j)g^{-1}(\lambda_i)\over \lambda_i^2+t\lambda_i\lambda_j+\lambda_j^2}\right)_{i,j=1,\cdots, n}
\end{align*}
is positive semidefinite.
\end{lemma}
\begin{theorem}\label{main2}
Let $A,B\in\mathcal{M}_n$  be positive semidefinite, $X\in\mathcal{M}_n$,  $f$, $g$ be two continuous functions on $(0,\infty)$ such that ${f(t)\over g(t)}$ is Kwong, and let $-2<t\leq2$. Then
 \begin{align}\label{hob5}
\omega\big(A^{1\over2}\big(H_{f,g}(A)\big)A^{1\over2}\big)\leq {2k\over t+2}\,\omega\big(A^2X+tAXA+XA^2\big),
 \end{align}
 where $k=\max_{\lambda\in\sigma(A)}\left\{{f(\lambda)g(\lambda)\over \lambda}\right\}$.\\
  Moreover,  inequality \eqref{hob5} is equivalent to the inequality
 \begin{align}\label{hob55}
\omega\big(A^{1\over2}\big(H_{f,g}(A,B)\big)B^{1\over2}\big)\leq {4k'\over t+2}\,\omega\big(A^2X+tAXB+XB^2\big),
 \end{align}
 where  $k'=\max_{\lambda\in\sigma(A)\cup\sigma(B)}\left\{{f(\lambda)g(\lambda)\over \lambda}\right\}$.
 \end{theorem}
 \begin{proof}
First, we show inequality \eqref{hob5}. It is enough to show the inequality in the case $A$ is positive definite. Since the numerical radius is weakly unitarily invariant, we may assume that $A$ is diagonal matrix with positive eigenvalues $\lambda_1, \cdots, \lambda_n$. Let $\Sigma={\rm diag}\left(\lambda_1^{1\over2}g(\lambda_1),\cdots,\lambda_n^{1\over2}g(\lambda_n)\right)$. It follows from Lemma \ref{fuji1} that
 \begin{align*}
Z=[z_{ij}]=\Sigma\Big({(t+2)\left(f(\lambda_i)g^{-1}(\lambda_j)+f(\lambda_j)g^{-1}(\lambda_j)\right)\over 2(\lambda_i^2+t\lambda_i\lambda_j+\lambda_j^2)}\Big)_{i,j=1,\cdots, n}\Sigma
\end{align*}
 is positive semidefinite for  $-2<t\leq2$. In addition, all diagonal entries of $Z$ are no more than $k$. Therefore,
 \begin{align*}
\|S_Z\|_\omega=\max_i z_{ii}=\max_i{f(\lambda_i)g(\lambda_i)\over \lambda_i} \leq k,
 \end{align*}
 whence ${\omega(Z\circ X)\over \omega(X)}\leq k\,\,(0\neq X \in\mathcal{M}_n)$.
 Now, let $M=\left[{1\over\lambda_i^2+t\lambda_i\lambda_j+\lambda_j^2}\right]_{i,j=1,\cdots, n}$ and $P=\left[{t+2\over 2}\lambda_i^{1\over2}f(\lambda_i)g(\lambda_j)+f(\lambda_j)g(\lambda_i)\lambda_j^{1\over2}\right]_{i,j=1,\cdots, n}$.
 Then
 \begin{align*}
 \omega(M\circ P\circ X)=\omega(Z\circ X) \leq k\,\omega(X)\qquad(0\neq X \in\mathcal{M}_n).
 \end{align*}
 Let  the matrix $N$ be the entrywise inverse of $M$, i.e., $M\circ N=J$. Hence
  \begin{align*}
 \omega(P\circ X)\leq k\,\omega(N\circ X)\qquad(0\neq X \in\mathcal{M}_n)
 \end{align*}
 or equivalently
  \begin{align*}
 \omega(A^{1\over2}\left(H_{f,g}(A)\right)A^{1\over2})\leq {2k\over t+2}\,\omega(A^2X+tAXA+XA^2),
 \end{align*}
 where $X\in\mathcal{M}_n$, $-2<t\leq2$  and $k=\max\left\{{f(\lambda)g(\lambda)\over \lambda}:\lambda\in\sigma(A)\right\}$. Hence we have inequality \eqref{hob5}.\\Now,  if we replace $A$ and $X$ by $\left(\begin{array}{cc} A&0\\ 0&B\end{array}\right)$ and $\left(\begin{array}{cc} 0&X\\ 0&0\end{array}\right)$ inequality \eqref{hob5}, respectively, then
 \begin{align*}
 \omega\left(\left[\begin{array}{cc} 0&A^{1\over2}\left(H_{f,g}(A,B)\right)B^{1\over2}\\ 0&0\end{array}\right]\right)
\leq {2k'\over t+2}\,\omega\left(\left[\begin{array}{cc} 0&A^2X+tAXB+XB^2\\ 0&0\end{array}\right]\right).
 \end{align*}
 Hence
\begin{align*}
 {1\over2}\omega\big(A^{1\over2}\big(H_{f,g}(A,B)\big)B^{1\over2}\big)&\leq\omega\left(\left[\begin{array}{cc} 0&A^{1\over2}\big(H_{f,g}(A,B)\big)B^{1\over2}\\ 0&0\end{array}\right]\right)\\&\qquad\qquad\qquad\qquad\qquad(\textrm{by Lemma}\,\ref{kit})
\\&\leq {2k'\over t+2}\,\omega\left(\left[\begin{array}{cc} 0&A^2X+tAXB+XB^2\\ 0&0\end{array}\right]\right)
\\&\leq {2k'\over t+2}\,\omega\big(A^2X+tAXB+XB^2\big)\,\,(\textrm{by Lemma}\,\ref{kit}).
 \end{align*}
 Thus, we reach inequality \eqref{hob55}. Also, if we put $B=A$ in inequality \eqref{hob5}, then we get inequality \eqref{hob55}.
\end{proof}

\begin{corollary}
Let $A\in\mathcal{M}_n$ be positive semidefinite. If $f$ is a positive operator monotone function on $(0,\infty)$, then
 \begin{align*}
\omega\big(A^{1\over2}f(A)Xf(A)^{-1}A^{3\over2}+A^{3\over2}f(A)^{-1}X&f(A)A^{1\over2}\big)\leq {4\over t+2}\,\omega\big(A^2X+tAXA+XA^2\big),
\end{align*}
 where $X\in\mathcal{M}_n$ and $-2<t\leq2$.
\end{corollary}
\begin{proof}
   Since $f$ positive operator monotone on $(0, \infty)$, then $g(t)=\frac{t}{f(t)}$ is operator monotone on $(0, \infty)$ and also $\frac{f(t)}{g(t)}=tf^2(t)$ is Kwong function \cite{ham}. So  $f$ and $g$ satisfy the conditions of Theorem \ref{main2}. Hence we have the desired inequality.
\end{proof}

\begin{example}
 The function $f(t)=\log(1+t)$ is operator monotone  on $(0,\infty)$; see \cite{mond}. If we put $g(t) =1$, then $\frac{f(t)}{g(t)}=\log(1+t)$ is Kwong \cite{kwong1}. Using Theorem \ref{main} we have
 \begin{align*}
\omega\big(A^{1\over2}\big(\log(I+A)X+X&\log(I+A)\big)A^{1\over2}\big)\leq {2\over t+2}\,\omega\big(A^2X+tAXA+XA^2\big),
 \end{align*}
 where $A, X\in\mathcal{M}_n$ such that $A$ is positive semidefinite  and  $-2<t\leq2$.
\end{example}
Now, we infer the following lemma due to Zhan \cite{Zhan}, which will be needed in the next theorem.
 \begin{lemma}\cite[Lemma 5]{Zhan}\label{fuji2}
Let $\lambda_1,\cdots, \lambda_n$ be any positive real numbers, $r\in[-1,1]$ and $-2<t\leq2$. Then the $n\times n$ matrix
\begin{align*}
L=\left({\lambda_i^r+\lambda_j^r\over \lambda_i^2+t\lambda_i\lambda_j+\lambda_j^2}\right)_{i,j=1,\cdots, n}
\end{align*}
is positive semidefinite.
\end{lemma}
Now, we shall show the following result related to \cite{fuji}.
\begin{proposition}\label{main3}
Let $A, X\in\mathcal{M}_n$ such that $A$ is positive semidefinite, $\beta>0$  and $1\leq 2r\leq3$. Then
 \begin{align*}
\omega\big(A^rXA^{2-r}+&A^{2-r}XA^r\big)\\&\leq\,\omega\left(2(1-2\beta+ 2\beta r_0)AXA + \frac{4\beta(1- r_0)}
{t + 2}(A^2X + tAXA + XA^2)\right),
 \end{align*}
 where  $-2 <t \leq 2\beta -2$ and $r_0=\min\{\frac{1}{2}+|1-r|, 1-|1 -r|\}$.
 \end{proposition}
\begin{proof}
Since the numerical radius is weakly unitarily invariant, we may assume that $A$ is diagonal matrix with positive eigenvalues $\lambda_1, \cdots, \lambda_n$.
Since  $1\leq2r\leq3 $, then $\frac{1}{2}\leq r_0\leq\frac{3}{4}$.
Let $t_0=\frac{1-2\beta+2\beta r_0}{2\beta(1-r_0)}(t+2) +t$. It follows from $-2 <t \leq 2\beta-2$ and $\frac{1}{4}\leq1-r_0\leq\frac{1}{4}$, that
 $\frac{t+2}{4\beta(1-r_0)}>0$ and $-2 <t_0\leq2$, where $t_0=\frac{t}{2\beta(1-r_0)}+\frac{1}{\beta(1-r_0)}-2$. Hence, by using Lemma \ref{fuji2}, the $n\times n$ matrix
 \begin{align*}
W=[w_{ij}]=\frac{t+2}{4\beta(1-r_0)}\Lambda^r\Big({\lambda_i^{2-2r}+\lambda_j^{2-2r}\over \lambda_i^2+t_0\lambda_i\lambda_j+\lambda_j^2}\Big)_{i,j=1,\cdots, n}\Lambda^r
\end{align*}
 is positive semidefinite for $\frac{1}{2}\leq r\leq\frac{3}{2}$, where $\Lambda={\rm diag}\left(\lambda_1,\cdots,\lambda_n\right)$.  Therefore,
 \begin{align*}
\|S_W\|_\omega=\max_i w_{ii}=\max_i{(t+2)\lambda_i^r(2\lambda_i^{2-2r})\lambda_i^r\over 4\beta(1-r_0)(t_0+2)\lambda_i^2}=1
 \end{align*}
 whence ${\omega(W\circ X)\over \omega(X)}\leq 1\,\,(0\neq X \in\mathcal{M}_n)$.
 Now, let $O=\left[{\lambda_i^2+t_0\lambda_i\lambda_j+\lambda_j^2}\right]_{i,j=1,\cdots, n}$ and $M=\left[\frac{1}{2(1-2\beta+ 2\beta r_0)\lambda_i\lambda_j + \frac{4\beta(1- r_0)}{t + 2}(\lambda_i^2X + t\lambda_i\lambda_j +\lambda_j^2)}\right]_{i,j=1,\cdots, n}$.
 Then
 \begin{align*}
 \omega(O\circ M\circ X)=\omega(W\circ X) \leq \,\omega(X)\qquad(0\neq X \in\mathcal{M}_n).
 \end{align*}
 Let  the matrix $N$ be the entrywise inverse of $M$, i.e., $M\circ N=J$. Hence
  \begin{align*}
 \omega(O\circ X)\leq \,\omega(N\circ X)\qquad(0\neq X \in\mathcal{M}_n)
 \end{align*}
 or equivalently
 \begin{align*}
\omega\big(A^rXA^{2-r}+&A^{2-r}XA^r\big)\\&\leq\,\omega\left(2(1-2\beta+ 2\beta r_0)AXA + \frac{4\beta(1- r_0)}
{t + 2}(A^2X + tAXA + XA^2)\right),
 \end{align*}
 where  $-2 <t \leq 2\beta -2$ and $r_0=\min\{\frac{1}{2}+|1-r|, 1-|1 -r|\}$. 
 \end{proof}

\bigskip
\bibliographystyle{amsplain}

\bigskip
\section{Acknowledgement.} The author would like to sincerely thank the referee for some useful comments and suggestions.

\end{document}